\documentclass[11pt,reqno]{amsart}
\usepackage{amsmath}
\usepackage{array}

\usepackage{mathabx}

\newtheorem{theorem}{Theorem}
\newtheorem{proposition}[theorem]{Proposition}
\newtheorem{lemma}[theorem]{Lemma}
\newtheorem{corollary}[theorem]{Corollary}

\theoremstyle{definition}
\newtheorem{definition}[theorem]{Definition}

\theoremstyle{remark}
\newtheorem{remark}[theorem]{Remark}
\def\L{{\mathcal L}}

\def\H{{\mathcal H}}
\def\P{{\mathcal P}}

\def\R{\Bbb R} 
\def\Z{\Bbb Z}

\def\C{\Bbb C}
\def\Q{\Bbb Q}

\def\be{\begin{equation}}
\def\ee{\end{equation}}
\def\bea{\begin{eqnarray}}
\def\eea{\end{eqnarray}}

\newcommand{\nn}{\nonumber\\}

\newcommand{\inv}{^{-1}}

\newcommand{\pbz}{\psi^\alpha_z}

\newcommand{\bbetap}{\bar{\beta'}}
\newcommand{\rl}{\rangle\langle}

\newcommand{\com}[1]
{{\overline{#1}}}

\def\bar{\overline}
\newcommand{\ec}[2]
{{\psi^{#1}_{#2}}}

\makeatletter

\@addtoreset{equation}{section}
\makeatother

\newcommand{\fact}{{\tilde g}}
\newcommand{\fract}{{g}}
\newcommand{\ccomp}[1]{\mathcal C(#1)}

\usepackage{hyperref}
\hypersetup{pdfborder=0 0 0}

\usepackage{color}

\begin{document}
\baselineskip=18pt
\title{On quantum complex flows}
\author{Thierry Paul}
\address{CNRS, 
LYSM – Laboratoire Ypatia de Sciences Mathématiques, Roma, Italia \& LJLL – Laboratoire Jacques-Louis Lions,  
Sorbonne Université 4 place Jussieu 75005 Paris France}
\email{thierry.paul@sorbonne-universite.fr}

\date{}

\begin{abstract}
We study the propagation of quantum T\"oplitz observables through quantized complex linear canonical transformation of one degree of freedom systems. We associate to such a propagated observable a non local ``T\"oplitz" expression involving off diagonal terms. We study the link of this constrauction with the usual Weyl symbolic paradigm.

%
%
%
%
%
\end{abstract}
\maketitle 

\small
\tableofcontents
\large

\section{Introduction and main result}\label{intmr}

Complex quantum Hamiltonians have (re)gained a lot of interest these last years, see e.g. the book \cite{Ler} and all the references quoted there. It seems that, at the contrary, quantization of complex symplectic flows didn't get the interest it deserves as it did at the period of the birth of Fourier integral operators, see \cite{Sjo} for example. Inside this category, representation of \textit{linear} complex symplectic flows, namely the complex symplectic group and its corresponding complex metaplectic representation has also lost most of interest since the golden years of ``group theoretical methods in physics", see \cite{gross}.

Let us remark that, since complex numbers are  (according to us) necessary to the formulation of quantum mechanics, considering the Schr\"odinger equation on a Hilbert space $\H$  with complex (e.g. bounded) Hamiltonian $H$ does not create any intrinsic a priori difficulty, a property that classical mechanics (on a real symplectic phase space $\P$) doesn't share with its quantum counterpart. The quantum flow is still given by $e^{-itH/\hbar}$ which exists (of course it is not anymore unitary) for all time for example when $H$ is bounded.  $e^{-itH/\hbar}$ defines for any $t\in\R$ a bounded operator on $\H$, at the contrary of the Lie exponential $\L^{-th}$ associated to the symbol $h$ of $H$ which doesn't apply on $\P$ because the Hamiltonian vector field associated to $h$ is not tangent to $\P$ anymore.

This complex setting is not the only one leading to to a situation where quantum mechanics is well defined while its underlying classical counterpart is not. Little regular potentials (or more generally Hamiltonians) generating vector fields below Cauchy-Lipschitz regularity condition are also examples of perfectively well defined quantum dynamics having a classical counterpart ill-defined, at least in the standard sense of this word. Another situation deals with classical chaotic systems, and the semiclassical approximation of its quantum counterpart in the case where the Planck constant $\hbar$ vanishes and  the time evolution diverges, the two limits being taken at the same time and being correlated. This twofold limit correspond to the classical (in the sense of $\hbar\to 0$) limit of the quantum flow taken at time infinite and could be naively though being the classical dynamics at time $t=\infty$, which doesn't exist, especially for chaotic systems.

We have been studied these singular (with respect to the classical underlying dynamics) situations in a series of articles dealing with the semiclassical approximation of the quantum dynamics.  In \cite{rough} was shown on examples  potentials giving rise to BV (and not more) vector field how quantum initial conditions of the Schr\"odinger equation select between the several\footnote{due to ill-definiteness of the dynamics not satisfying Cauchy-Lipschitz condition} 
possible classical trajectories the one obtained at the classical limit (``superposition of them are also reachable). In \cite{ambfigal,figal} is shown in the general BH case how the Di Perna-Lions \cite{diperna} classical flow associated to these singular Hamiltonians is recovered at the classical limit, see \cite{p2} for a (short) review of results concerning singular potentials. Long time semiclassical ``chaotic" evolution was first apprehended  in \cite{p}, where splitting and reconstruction of evolved coherent states was shown to happen around hyperbolic fixed points of the classical dynamics, and is studied more systematically in \cite{p1} where we show that the limit $\hbar\to 0,t\to\infty$ leads to non standard classical limit : the phase space - locus where the symbol of the evolved observable is defined -  becomes a noncomuutative space (space of leaves of the invariant foliation of the classical flow). Finally in \cite{p3} we studied ``quantum" observables constructed in the framework of topological quantum fields theory involving  even more singular (below continuous) naive symbols. A change of paradigm dealing with operator valued symbols is introduced in order to define the ``right" symbol (and the right underlying phase-space).
 
 In all these papers, the phase space (and therefore the dynamics) obtained by taking the classical limit $\hbar\to 0$ has had to be changed from the  (standar symplectig manifold) expected ones: \cite{p,rough} the limit dynamics becomes probabilistic (or ubiquitous), in \cite{ambfigal,figal} the limit flow is only defined almost everywhere and in \cite{p1,p3} the phase space becomes a noncommutative space. Moreover in \cite{p1,p3}, though it is the standard quantum dynamics which is studied, new types of quantizations were needed both in a form of, say, non local T\"oplitz quantization.
 The goal of the present little paper is to show how the \textit{standard} quantization of complex linear canonical transformations can be understood in terms of \textit{real} flows (actually two) and a genuine change of type of quantization.
 \vskip 1cm The Dirac notation will be used through the whole article:  $\vert\#\rangle$ will be meant as an element of a Hilbert space $\mathcal H$ with scalar product $(\cdot,\cdot)$, $\langle\#1\vert\#2\rangle$ as the scalar product 
$(\vert\#1\rangle,\Vert\#2\rangle)$, and $\vert\#1\rangle\langle\#2\vert$ as the operator on $\mathcal H$ defined by $\vert\#1\rangle\langle\#2\vert\varphi:=(\vert\#2\rangle,\varphi)\vert\#1\rangle.$, $\varphi\in\mathcal H$.
 
 \vskip 1cm
We will study conjugation of quantum observables given by the  T\"oplitz (anti-Wick) construction by complex metaplectic operators. In order to avoid heavy notation, we will be, in most of the paper, concern with the one degree of freedom case: a two dimensional phase space that will be, for simplicity, $\C$.


 More precisely we consider operators of the form 
\[
H=\int_\C h_\alpha(z)\vert\pbz\rangle\langle\pbz\vert\frac{dzd\bar z}{2\pi\hbar}.
\]
where the family of coherent states $\pbz$ is defined, 
for $\alpha\in\C,\ \Im\alpha>0$ and $z=(q,p)\in\R^2$ the coherent state $\psi^\alpha_z$ defined by
\[
\psi^\alpha_z(x)
:=
\left(\frac{\Im\alpha}{\pi\hbar|\alpha|^2}\right)^{\frac14}e^{-\frac i{2\hbar\alpha}(x-q)^2}e^{i\frac{px}\hbar}e^{-i\frac{pq}{2\hbar}}.
\]
Note that the standard T\"oplitz quantization correspond to $\beta=1$, but $\beta$ will show to be a true dynamical variable, so we need to consider it as a degree of freedom. the link between different $h_\beta(\cdot)$ leading to the same operator $H$ (in particular the standard case) is given below, in Appendix \ref{w}, Lemma \ref{hushus}. One has
\begin{eqnarray}
\int_\C h_\alpha(z)\vert\pbz\rangle\langle\pbz\vert\frac{dzd\bar z}{2\pi\hbar}
&=&
\int_\C h_{\alpha,\alpha'}(z)\vert\psi^{\alpha'}_z\rangle\langle\psi^{\alpha'}_z\vert\frac{dzd\bar z}{2\pi\hbar}\label{alphaalphaprime}\\
&\Updownarrow&\nonumber\\
h_{\alpha,\alpha'}&=&e^{-i\frac\hbar4(\alpha-\alpha')\Delta_\xi+i\frac\hbar4(\frac1\alpha-\frac1{\alpha'})\Delta_x}h_{\alpha'}\label{betabetaprime}
\end{eqnarray}
for $ \Im{(\alpha-\alpha')}>0.\ \Im{(\frac1\alpha-\frac1{\alpha'})}<0$.

An easy computation shows that, when $h_\alpha(z):=1$, $H=\mbox{Id}$ for any value of $\alpha$.
\vskip 1cm
We consider the operator $U(S)\inv HU(S)$, conjugated of $H$ by the operator $U(S)$ where $S$ is a real $2\times 2$ matrix $\begin{pmatrix}
a&b\\c&d
\end{pmatrix} $ of determinant one and $U$ is the metaplectic representation. More precisely $U(S)$ is the operator of integral kernel given by
\bea\label{udexy}
U(x,y)&=&\frac1{\sqrt{b2\pi\hbar}} e^{-\frac i{2b\hbar}\left(dx^2-2xy+ay^2\right)}\ \ \ \ \ \ \ \ b\neq 0\nonumber\\
&&\\
&=&{\sqrt d}\delta\left(dx  -y\right)e^{i\frac c d\frac{x^2}2}\ \ \ \ \  \ \ \ \ \ \ \ \ \ \ \   b=0\nonumber
\eea
In order to show of \eqref{udexy} can be derived, let us recall that one way to define the metapletic representation is through the formula
\be\label{meta}
S\binom{x}{-i\hbar\frac d{dx}}=\binom{U(S)^{-1}xU(S)}{U(S)^{-1}{(-i\hbar\frac d{dx})}U(S)}.
\ee
Writing that $U(S)$ is unitary by $U(S)^{-1}=U(S)^*$, that is $U^{-1}(x,y)=\bar{U(y,x)}$ we get for $S=\begin{pmatrix}
a&b\\c&d\end{pmatrix}$, an equation whose solution is \eqref{udexy}, of course modulo a global phase.

  It is well known and easy to derive after \eqref{meta} that, when $S$ is real the Weyl symbol of $U(S)HU(S)^{-1}$ is the push-forward of the Weyl symbol of $H$ by $S$. namely
  \be\label{pfw}
  \sigma_{U(S)^{-1}HU(S)}^{Weyl}\binom{q}{p}=\sigma_{H}^{Weyl}\left(S^{-1}\binom{q}{p}\right).
  \ee
  
  An easy computation
  shows that this implies the following result.
  \[
U(S)^{-1}HU(S)=\int h_{\alpha_S(\alpha),\alpha}(Sz)\vert\pbz\rangle\langle\pbz\vert\frac{dzd\bar z}{2\pi\hbar}
\]
$\mbox{ with }\alpha_S(\alpha)=iS\cdot(\alpha)$, where  $S\cdot z=\frac{az+b}{cz+d}$, and $h_{\alpha_S(\alpha),\alpha}$ is given by \eqref{betabetaprime}.
 
 Of course this formula doesn't make any sense when $S$ is not real any more for general symbols $h$. But we will see that there is a general ``off-diagonal" T\"oplitz representation.

The main result of the present short note is the following theorem.

\textbf{Notation}: 
we will denote 
for $S= \begin{pmatrix}
a&b\\c&d
\end{pmatrix}$, $z=(q,p)$ and $\alpha\in\C$, 
$$
\com{S}= \begin{pmatrix}
\bar  a&\bar b\\\bar c&\bar d
\end{pmatrix},\ S (z)=\begin{pmatrix}
a&b\\c&d
\end{pmatrix}\begin{pmatrix}
q\\p
\end{pmatrix}\mbox{ and }
S\cdot\alpha=\tfrac{a\alpha+b}{c\alpha+d}.
$$
Moreover, $\wedge$ will denote the symplectic form pn $T^*\R$, $z\wedge z'=pq'-qp'$.
\begin{theorem}[Off-diagonal T\"oplitz representation]\label{thm1}\ 

Let
\[
H=\int h(z)\vert\psi^\alpha_z\rangle\langle\psi^\alpha_z\vert\frac{dzd\bar z}{2\pi\hbar}.
\]
Define, for $S
\in SL(2,\Z)$ and $\Im\alpha>0$, the real $2\times 2$ matrix ${_\alpha T_S}$ by 
$$
{_\alpha T_S(z)}=(q_S^\alpha,p_S^\alpha)\in\R^2\mbox{ defined by }q_S^\alpha+\alpha_S(\alpha)p_S^\alpha=q^S+\alpha_S(\alpha)p^S,\ 
\begin{pmatrix}
q^S\\p^S
\end{pmatrix}:=S(z),\ z=(q,p).
$$
Then, for any $S,\alpha$ such that $|V\cdot \alpha|<\infty,\ \Im(V\cdot \alpha)>0,\ V=S^{-1},\com{S}^{-1}, S^{-1}\com{S}$, 
\[
U(S)^{-1}HU(S)=\int 
{_\alpha T_{\com{S}}}\#h_{\com{S}\cdot \alpha,\alpha}(z)
\tfrac{\vert\psi^{S^{-1}\com{S}\cdot\alpha}
_{{_\alpha T_{S^{-1}\com{S}}}(z)}
\rangle\langle\psi^{\alpha}_z\vert}
{\langle\psi^{S^{-1}\com{S}\cdot\alpha}
_{{_\alpha T_{S^{-1}\com{S}}}(z)}
\vert\psi^{\alpha}_z\rangle}
\frac{dzd\bar z}{2\pi\hbar},
\]
where $h_{\com{S}\cdot \alpha,\alpha}$ is defined by \eqref{alphaalphaprime}-\eqref{betabetaprime} and $T\# h$ is the push-forward of $h$ by $T$.

\end{theorem}
\begin{remark}\label{rm0}
It is easy to check that, for any $S\in SL(2,\C)$ there exists $\alpha,\ \Im\alpha>0,$ such that $|V\cdot \alpha|<\infty,\ \Im(V\cdot \alpha)>0,\ V=S^{-1},\com{S}^{-1}, S^{-1}\com{S}$, 
\end{remark}
\begin{remark}\label{rm1}
We could have thought to try to use $U(S)^{-1}HU(S)\psi ^\beta_z$ for $H$ pseudo, but this creates $O(\hbar^\infty)$ terms in competition with the large, as $\hbar\to 0$, terms of the norm of $U(S)\psi^\beta_z$.
\end{remark}
\begin{remark}\label{rm2}
Our main result Theorem \ref{meta} shows that, associated to the  complex linear canonical mapping (flow) $S$ are associated two real (non flow) linear mappings: ${_\alpha T_\com{S}}$ which propagate the T\"oplitz symbol, and .${_\alpha T_{S^{-1}\com{S}}}$ which ``propagate" the off-diagonal property of the T\"oplitz quantization
\end{remark}
\vskip 1cm


We computed in Section \ref{annb} several  simple examples.

\section{Proof of Theorem \ref{thm1}}\label{proof}
We state first the following well-know (and trivial to prove) result.
\begin{lemma}\label{lem1}
Let $S=\left(\begin{matrix}
a&b\\
c&d
\end{matrix}\right)\in SL(2,\R)=Spn(2,\R)$. We define $U=U(S)$ the operator on $L^2(\R)$ defined through its integral kernel
\bea
U(x,y)&=&\frac1{\sqrt{b2\pi\hbar}} e^{-\frac i{2b\hbar}\left(dx^2-2xy+dy^2\right)}\ \ \ \ \ \ \ \ \ \ b\neq 0\nonumber\\
&=&{\sqrt a}\delta(a x  -y)e^{-i\frac c a\frac{x^2}2}\ \ \ \ \ \ \ \ \ \ \ \ \ \ \  b=0\nonumber
\eea
(remark that $U$ is continuous as an opertaor as $b\to 0$).
Then $U$ is unitary and 
\[
U(S)^{-1}\left(\begin{matrix}
P\\Q\end{matrix}\right)U(S)=S\left(\begin{matrix}
P\\Q\end{matrix}\right)
\]
where $Q=\times x,\ P=-i\hbar\frac d{dx}$.
\end{lemma}
\begin{proposition}\label{lem2}
Let, for $\alpha\in\C,\ \Im\alpha>0$ and $z=(q,p)\in\R^2$, the coherent state $\psi^\alpha_z$ be defined by
\[
\psi^\alpha_z(x)
:=
\left(\frac{\Im\alpha}{\pi\hbar|\alpha|^2}\right)^{\frac14}e^{-\frac i{2\hbar\alpha}(x-q)^2}e^{i\frac{px}\hbar}e^{-i\frac{pq}{2\hbar}}.
\]
Let $S=\left(\begin{matrix}
a&b\\
c&d
\end{matrix}\right)\in SL(2,\C)=Spn(2,\C)$ and let  $U=U(S)$ be defined through its integral kernel
\bea
U(x,y)&=&\frac1{\sqrt{b2\pi\hbar}} e^{-\frac i{2b\hbar}\left(dx^2-2xy+ay^2\right)}\ \ \ \ \ \ \ \ \ \ b\neq 0\nonumber\\
&=&{\sqrt d}\delta(d x  -y)e^{i\frac c d\frac{x^2}2}\ \ \ \ \  \ \ \ \ \ \ \ \ \ \ \ \ \  b=0\nonumber
\eea
Let $\wedge$ be the symplectic form $z\wedge z'=pq'-qp'$ and
\be\left\{
\begin{array}{l}
z^S=(p^S,q^S),\ \binom{p^S}{q^S}=S\binom{p}{q},\\
\alpha_S(\alpha)=S\cdot\alpha,\ \begin{pmatrix}
a&b\\c&d
\end{pmatrix}\cdot\theta=\tfrac{a\theta+b}{c\theta+d},\ \theta\in\C\\
z_S^\alpha=: {_\alpha T_S(z)}=(p_S^\alpha,q_S^\alpha)\in\R^2\mbox{ defined by }q_S^\alpha+\alpha_S(\alpha)p_S^\alpha=q^S+\alpha_S(\alpha)p^S\\
{_\alpha\fact_S}=(a+b\alpha^{-1})^{-\frac12}\left(\tfrac{\Im(S\cdot\alpha)^{-1}}{\Im(\alpha^{-1})} \right)^{\frac14}.
\end{array}
\right.
\ee
Then, for any $S,\alpha$ such that $|\alpha_S(\alpha)|<\infty,\ \Im\alpha_S(\alpha)>0$, 
\[
U(S)\psi^\alpha_{z}=e^{i\frac {z^S\wedge z_S^\alpha}{2\hbar}}\psi^{\alpha_S(\alpha)}_{z_S^\alpha}.
\]
\end{proposition}
\begin{proof}
We first notice that, denoting $Z=\binom{P}{Q}$,
$$
\pbz=e^{i\frac{z\wedge Z}\hbar}\psi^\alpha_0.
$$
Therefore
\begin{eqnarray}
U(S)\pbz&=&
U(S)e^{i\frac{z\wedge Z}\hbar}U(S)^{-1}U(S)\psi^\alpha_0\nonumber\\
&=&
e^{i\frac{z\wedge S^{-1}(Z)}\hbar}U(S)\psi^\alpha_0\nonumber\\
&=&
e^{i\frac{S(z)\wedge (Z)}\hbar}U(S)\psi^\alpha_0\nonumber\\
&=&
e^{i\frac{S(z)\wedge (Z)}\hbar}
(a+b\alpha^{-1})^{-\frac12}\left(\tfrac{\Im(S\cdot\alpha)^{-1}}{\Im(\alpha^{-1})} \right)^{\frac14}
\psi^{S\cdot\alpha}_0\nonumber\\
&=&
e^{i\frac{S(z)\wedge (Z)}\hbar}
{_\alpha\fact_S}
\psi^{S\cdot\alpha}_0\nonumber
\end{eqnarray}
(we have used that, since $S$ is canonical, $z\wedge S^{-1}(Z)=S^{-1}(S(z))\wedge S^{-1}(Z)=S^{-1}(z)\wedge Z$ and $U(S)\psi^\alpha_0=\psi^{S\dot \alpha}_0$ as an straightforward computation shows.

Now
\begin{eqnarray}
e^{i\frac{S(z)\wedge (Z)}\hbar}\psi^{S\cdot\alpha}_{0}(x)&=&
\left(\frac{\Im\alpha_S(\alpha)}{\pi\hbar|\alpha_S(\alpha)|^2}\right)^{\frac14}e^{-\frac i{2\hbar\alpha_S(\alpha)}(x-q^S)^2}e^{i\frac{p^Sx}\hbar}e^{-i\frac{p^Sq^S}{2\hbar}}
\nonumber\\
&=&
e^{iC}\left(\frac{\Im\alpha_S(\alpha)}{\pi\hbar|\alpha_S(\alpha)|^2}\right)^{\frac14}e^{-\frac i{2\hbar\alpha_S(\alpha)}(x-q_S)^2}e^{i\frac{p_Sx}\hbar}e^{-i\frac{p_Sq_S}{2\hbar}},\ p_S,q_S\in\R\nonumber\\
&=&
e^{iC}\psi^{S\dot\alpha}_z\nonumber
\end{eqnarray}
at the condition that, precisely,
$$
\left\{
\begin{array}{l}
q_S^\alpha+\alpha_S(\alpha)p_S^\alpha=q^S+\alpha_S(\alpha)p^S\\
C=\frac1{2\hbar}\left(\frac{q_S^2.(q^S)^2}{\alpha_S(\alpha)}+(p_Sq_S-p^Sq^S\right)=\frac{\binom{p^S}{q^S}\wedge\binom{p_S}{q_S}}{2\hbar}\mbox{ by using the preceding equality.}
\end{array}
\right.
$$
\end{proof}
The proof of Theorem \ref{meta} follows easily by simple computations after first remarking that
$$
U(S^{-1})\vert\pbz\rl\pbz\vert U(S)=
U(S^{-1})\vert\pbz\rl U(S)^*\pbz\vert=
\vert U(S^{-1})\pbz\rl U(\com{S}^{-1})\pbz\vert 
$$
Then we use Proposition \ref{lem2} in order to express the last term in the preceding equalities:
$$
\vert U(S^{-1})\pbz\rl U(\com{S}^{-1})\pbz\vert =
C\vert\psi^{S^{-1}\cdot\alpha}_{ _\alpha T_{{S}^{-1}}z}\rangle\langle
\psi^{\com{S}^{-1}\cdot\alpha}_{ _\alpha T_{\com{S}^{-1}}z}
\vert; \ C\in\C.
$$
 Note that there is no need to consider out of Proposition \ref{lem2} the exact value of the constant $C$ thanks to the following following trick: one easily compute that
$$
U(S)^{-1}\vert\psi^\alpha_z\rangle\langle\psi^\alpha_z\vert U(S)=L
\vert\psi^{S^{-1}\cdot\alpha}_{_\alpha T_{S^{-1}}z}\rangle\langle
\psi^{{\com{S}}^{-1}\cdot\alpha}_{_\alpha T_{{\com{S}}^{-1}}z}\vert \mbox{ for } L\in\C.
$$
But $U(S)^{-1}\vert\psi^\alpha_z\rangle\langle\psi^\alpha_z\vert U(S)$ is a projector, therefore 
$$
L^2\langle \psi^{S^{-1}\cdot\alpha}_{_\alpha T_{S^{-1}}z}\vert
\psi^{{\com{S}}^{-1}\cdot\alpha}_{_\alpha T_{{\com{S}}^{-1}}z}\rangle=L
\Longrightarrow L=\tfrac1{\langle \psi^{S^{-1}\cdot\alpha}_{_\alpha T_{S^{-1}}z}\vert
\psi^{{\com{S}}^{-1}\cdot\alpha}_{_\alpha T_{{\com{S}}^{-1}}z}\rangle}.
$$
 Performing  finally in the ``T\"oplitz integral" te change of variable $z\to{_\alpha T_\com{S}}(z)$ gives the result. Theorem \ref{thm1} is proved.

\section{Examples}\label{annb}

Several examples are presented in the table below.

Let
$$
D_\alpha(q,p)=
{\langle\psi^{S^{-1}\com{S}\cdot\alpha}
_{{_\alpha T_{S^{-1}\com{S}}}(z)}
\vert\psi^{\alpha}_z\rangle}^{-1},\ z=(q,p).
$$
\vskip 1cm
\begin{tabular}
{|c|c|c|c|c|c|c|}
\hline  
\tiny \begin{color}{white}$\begin{pmatrix}\\\end{pmatrix}$\end{color}\tiny $U$\tiny \begin{color}{white}$\begin{pmatrix}\\\end{pmatrix}$\end{color}&\tiny $S$ &\tiny $S^{-1}S^c$ &\tiny $S^{-1}S^c\cdot i$ &\tiny $_iT_{S^{-1}S^c}$ &\tiny 
$_iT_{S^c}$ &\tiny $D_i(q,p)$\\  \hline

\tiny \begin{color}{white}$\begin{pmatrix}.\\.\end{pmatrix}$\end{color}  \tiny   free ev.
\tiny \begin{color}{white}$\begin{pmatrix}.\\.\end{pmatrix}$\end{color} 
&\ \tiny $\ \ \begin{pmatrix}1&-it\\0&1\end{pmatrix}\ \ $   &\ \tiny $\ \ \begin{pmatrix}1&2it\\0&1\end{pmatrix}\ \ $   &\tiny $ i(1+2t)$ &\ \tiny $\ \ \begin{pmatrix}1&0\\0&\tfrac{1+4t}{1+2t}\end{pmatrix}\ \ $    &\ \tiny $\ \ \begin{pmatrix}1&0\\0&\tfrac{1+2t}{1+t}\end{pmatrix}\ \ $   &\tiny $e^{t\frac{(1+3t)^2}{(1+2t)^2}\frac{p^2}\hbar}$\\   \hline

\tiny \begin{color}{white}$\begin{pmatrix}.\\.\end{pmatrix}$\end{color}\tiny  $\times $ by  $  e^{-tx²} $\tiny\begin{color}{white}$\begin{pmatrix}.\\.\end{pmatrix}$\end{color} &\ \tiny $\ \ \begin{pmatrix}1&0\\-it&1\end{pmatrix}\ \ $   &\ \tiny $\ \ \begin{pmatrix}1&0\\2it&1\end{pmatrix}\ \ $   &\tiny $\frac i{1-2t}$ &\ \tiny $\ \ \begin{pmatrix}\tfrac{1-4t}{1-2t}&0\\0&1\end{pmatrix}\ \ $    &\ \tiny $\ \ \begin{pmatrix}\tfrac{1-2t}{1-t}&0\\0&1\end{pmatrix}\ \ $   &\tiny $e^{t\frac{(1+3t)^2}{(1+2t)^2}\frac{q^2}\hbar}$\\   \hline

\tiny\begin{color}{white}$\begin{pmatrix}.\\.\end{pmatrix}$\end{color}\tiny   dilation \tiny\begin{color}{white}$\begin{pmatrix}.\\.\end{pmatrix}$\end{color} &\ \tiny $\ \ \begin{pmatrix}e^{it}&0\\0&e^{-it}\end{pmatrix}\ \ $   &\ \tiny $\ \ \begin{pmatrix}e^{-2it}&0\\0&e^{2it}\end{pmatrix}\ \ $   &\tiny $e^{-4it}i$ &\ \tiny $\ \ \begin{pmatrix}\tfrac{1-4t}{1-2t}&0\\0&1\end{pmatrix}\ \ $    &\ \tiny $\ \ \begin{pmatrix}\tfrac{1-2t}{1-t}&0\\0&1\end{pmatrix}\ \ $   &\tiny $e^{2i\cos t\frac{ qp}\hbar}$\\   \hline

\tiny\begin{color}{white}$\begin{pmatrix}.\\.\end{pmatrix}$\end{color}\tiny   oscillator\tiny\begin{color}{white}$\begin{pmatrix}.\\.\end{pmatrix}$\end{color} &\ \tiny $\ \ \begin{pmatrix}\cosh t&i\sinh t\\-i\sinh t&\cosh t\end{pmatrix}\ \ $   &\ \tiny $\ \ \begin{pmatrix}\cosh 2t&i\sinh 2t\\-i\sinh 2t&\cosh 2t\end{pmatrix}\ \ $   &\tiny $i$ &\ \tiny $\ \ \begin{pmatrix}e^{-2t}&0\\0&e^{2t}\end{pmatrix}\ \ $    &\ \tiny $\ \ \begin{pmatrix}e^{-t}&0\\0&e^{t}\end{pmatrix}\ \ $    &\tiny $e^{\frac{\sinh 2t(q^2+p^2)}\hbar}$\\   \hline

\tiny\begin{color}{white}$\begin{pmatrix}.\\.\end{pmatrix}$\end{color}\tiny $\dots$\tiny\begin{color}{white}$\begin{pmatrix}.\\.\end{pmatrix}$\end{color} &\ \tiny $\ \ \begin{pmatrix}0&i\\i&0\end{pmatrix}\ \ $   &\ \tiny $\ \ \begin{pmatrix}-1&0\\0&-1\end{pmatrix}\ \ $   &\tiny $i$ &\ \tiny $\ \ \begin{pmatrix}-1&0\\0&-1\end{pmatrix}\ \ $    &\ \tiny $\ \ \begin{pmatrix}-1&0\\0&1\end{pmatrix}\ \ $   &\tiny $1$\\   \hline

\end{tabular}


\vskip 1cm
We finish this section by the same computations for a non-canonical $S$, used in Section \ref{nomcatrans}  below.
\vskip 1cm
\begin{tabular}
{|c|c|c|c|c|c|c|}
\hline  
\tiny \begin{color}{white}$\begin{pmatrix}\\\end{pmatrix}$\end{color}\tiny \tiny \begin{color}{white}$\begin{pmatrix}\\\end{pmatrix}$\end{color}&\tiny $S$ &\tiny $S^{-1}S^c$ &\tiny $S^{-1}S^c\cdot i$ &\tiny $_iT_{S^{-1}S^c}$ &\tiny 
$_iT_{S^c}$ 
\\  \hline


\tiny\begin{color}{white}$\begin{pmatrix}.\\.\end{pmatrix}$\end{color}\tiny its opposite\tiny\begin{color}{white}$\begin{pmatrix}.\\.\end{pmatrix}$\end{color} &\ \tiny $\ \ \begin{pmatrix}0&i\\-i&0\end{pmatrix}\ \ $   &\ \tiny $\ \ \begin{pmatrix}1&0\\0&1\end{pmatrix}\ \ $   &\tiny $i$ &\ \tiny $\ \ \begin{pmatrix}1&0\\0&1\end{pmatrix}\ \ $    &\ \tiny $\ \ \begin{pmatrix}-1&0\\0&-1\end{pmatrix}\ \ $   
\\   \hline

\tiny\begin{color}{white}$\begin{pmatrix}.\\.\end{pmatrix}$\end{color}\tiny anticanonical example\tiny\begin{color}{white}$\begin{pmatrix}.\\.\end{pmatrix}$\end{color} &\ \tiny $\ \ \begin{pmatrix}0&-i\\i&0\end{pmatrix}\ \ $   &\ \tiny $\ \ \begin{pmatrix}-1&0\\0&-1\end{pmatrix}\ \ $   &\tiny $i$ &\ \tiny $\ \ \begin{pmatrix}-1&0\\0&-1\end{pmatrix}\ \ $    &\ \tiny $\ \ \begin{pmatrix}1&0\\0&1\end{pmatrix}\ \ $   
\\   \hline

\end{tabular}



\section{Link with Weyl}
\begin{theorem}
Let $H$ with T\"oplitz symbol $h$. Let us denote by $\sigma^\hbar_{H}$ the Weyl symbol of $H$ (note that $\sigma^\hbar_{H}$ is an entire function). Then
\[
\sigma^\hbar_{U(S)\inv HU(S)}=\sigma^\hbar\circ S
\]
and
\[
\sigma^0_{U(S)\inv HU(S)}(q,p)=\int \widehat{h}(\xi,x)e^{-i(q^S\xi+p^Sx)}dxd\xi.
\]
\end{theorem}
\section{Flows on extended phase-space}
Consider on the extended phase space $\P:=T^*\R\times \C^+=\{(z,\beta)\}$ the mapping
\[
\Phi_S: (z,\alpha)\mapsto ({_\alpha T_S}(z),S\cdot\alpha).
\].
One proves easily the following resulst.
\begin{theorem}\label{flot}
\[
\Phi_{SS'}=\Phi_S\Phi_{S'}.
\]
and $S\to\Phi_S$ is a representation of $SP(2n,\C)$.
\end{theorem}

\section{Noncommutative geometry interpretation }
\label{noncoi}In this section we give a noncommutative interpretation of the off diagonal Toeplitz representation in Theorem \ref{thm1}.
\subsection{The canonical groupoid}\label{cangroup}\ 

We consider on $\P=T^*\R\times\C^+$ the action of the group $SP(2n,\C)$ 
defined, for any $S\in SP(2n,\C)$ by
$$
(z,\alpha)\to({_\alpha T_S}(z),S\cdot\alpha).
$$
Let us define the groupoid $G$ defined as the semi-direct product $\P\rtimes  SL(2n,\C)$ of $\P$ by $SL(2n,\C)$ \cite[Definition1 p. 104-105 and Section 7]{ac} as  $G=\P\times SL(2n,\C)$, $G^{(0)}=\P\times\{\bf 1\}$ and the functors range and source given by
$$
r((z,\alpha),S)=(z,\alpha),\ \ \ s((z,\alpha),S)=((S(z),S\cdot\alpha)\ \ \forall ((z,\alpha),S)\in\P\times SL(2n,\R).
$$
The $C^*$ algebra associated to the groupoid $G$ is the crossed product $C_0(\P)\rtimes_\Phi SL(2n,\C)$ of the algebra of continuous functions on $\P$ by the action of $SL(2n,\C)$ defined by $\Phi$.
\subsection{Symbols}\label{symb}\ 

Let $H$ be a ${_\alpha \mbox{T}}$\"oplitz operator of symbol $\sigma^T_H$ as given by \eqref{alphaalphaprime}. By Theorem \ref{thm1}, we associate to $U(S)^{-1}HU(S)$ the 
couple
$$
((\sigma^T_H)_{S\cdot\alpha,\alpha}\circ {\alpha T}_{\com{S}},\Phi_{S^{-1}\com{S}})
$$
where $(\sigma^T_H)_{S\cdot\alpha,\alpha}$ is given by \eqref{betabetaprime}.

This can be seen as an  element of the algebra associated to the canonical groupoid defined in Section \ref{cangroup} by the following construction:
 we associate to $((\sigma^T_H)_{S\cdot\alpha,\alpha}\circ {\alpha T}_{\com{S}},\Phi_{S^{-1}\com{S}})$ the function $\sigma^{off}[U(S)^{-1}HU(S)]$  on the canonical groupoid identified with $\P\times\P$ defined by
\begin{eqnarray}
&&\sigma^{off}[U(S)^{-1}HU(S)]((z,\alpha),(z',\alpha'))
:=
{_\alpha T}_{\com{S}}\#(\sigma^T_H)_{S\cdot\alpha,\alpha}(z))\delta((z',\alpha')-\Phi_{S^{-1}\com{S}}(z,\alpha))
,\nonumber
\end{eqnarray}
where ${_\alpha T}_{\com{S}}\#(\sigma^T_H)_{S\cdot\alpha,\alpha}$ designate the push-forward of $(\sigma^T_H)_{S\cdot\alpha,\alpha}$ by ${_\alpha T}_{\com{S}}$.

Conversely,  we ``quantize" the symbol $\sigma_{U(S)^{-1}HU(S)}^{off}$ by the following off-diagonal Toeplitz type quantization formula
\be\label{formquant}
T^{off}[\sigma
^{off}]
:=
\int_{\P\times\P}\sigma^{off}
((z,\alpha),(z',\alpha'))
\tfrac{|\psi^{\alpha'}_{z'}\rangle\langle\psi^\alpha_z|}
{\langle\psi^{\alpha'}_{z'}|\psi^\alpha_z\rangle}
\tfrac{dzd\bar zdz'd\bar{z'}}{2\pi\hbar}.
\ee
\begin{proposition}
$$
T^{off}[\sigma^{off}[U(S)^{-1}HU(S)]]=U(S)^{-1}HU(S).
$$
\end{proposition}
\subsection{On the (formal) composition of symbols}\label{compsymb}\ 

Conjugating an observable by $U(S)$ correspond classically to a (complex or real) change of variable in the classical underlying paradigm. Therefore, multiplication of functions should be defined on the same system of coordinates, computationally. This leads to  associate to $U(S)^{-1}HU(S)$ the operator of ``multiplication"  acting on $H'$ given by
\begin{eqnarray}
U(S)^{-1}HU(S)\cdot_S H'&:=&
U(S)^{-1}HH'U(S).\nonumber
\end{eqnarray}

This gives rise to the following multiplication of symbols: when $H,H'$ are Toeplitz operators, so is (asymptotically) $HH'$ and its symbol is at leading order the product of the symbol of $H$ by the one of $H'$. Therefore the symbol of $U(S)^{-1}HU(S)\cdot_S H'$ is the groupoid composition of the one of $U(S)^{-1}HU(S)$ by th (trivial) one of $H'$. 

In the case where $H':=U(S')^{-1}H^{in}U(S')$

\begin{eqnarray}
U(S)^{-1}HU(S)\cdot_S H'&:=&
U(S)^{-1}HH'U(S).\nonumber
\\
&=&U(S)^{-1}HU(S')^{-1}H'U(S')U(S).\nonumber
\end{eqnarray}
using the result of Theorem \ref{thm1}
\begin{eqnarray}
&&U(S')^{-1}HU(S')\nonumber\\
&=&\int 
{_\alpha T_\com{S'}}\#h_{\com{S'}\cdot \alpha,\alpha}(_\alpha T_{\com{S'}}z)
\tfrac{\vert\psi^{{S'}^{-1}\com{{S'}}\cdot\alpha}
_{{_\alpha T_{{S'}^{-1}\com{{S'}}}}(z)}
\rangle\langle\psi^{\alpha}_z\vert}
{\langle\psi^{{S'}^{-1}\com{{S'}}\cdot\alpha}
_{{_\alpha T_{{S'}^{-1}\com{{S'}}}}(z)}
\vert\psi^{\alpha}_z\rangle}
\frac{dzd\bar z}{2\pi\hbar},
\nonumber
\end{eqnarray}
and  (formally)
$$
H\vert\psi^{{S'}^{-1}\com{{S'}}\cdot\alpha}
_{{_\alpha T_{{S'}^{-1}\com{{S'}}}}(z)}
\rangle=h({{_\alpha T_{{S'}^{-1}\com{{S'}}}}(z)})
\vert\psi^{{S'}^{-1}\com{{S'}}\cdot\alpha}
_{{_\alpha T_{{S'}^{-1}\com{{S'}}}}(z)}
\rangle +O(\hbar)
$$
we get formally the usual groupoid composition of symbols.
 
\section{Non canonical transforms}
\label{nomcatrans}
%
%

It is striking to notice that the definition of the metaplectic representation as defined by \eqref{udexy}, namely, for a matrix $S=\begin{pmatrix}
a&b\\c&d
\end{pmatrix} $ of determinant one,
 the operator of integral kernel given by \eqref{udexy},
depends only on the numbers $a,b,d$. The absence of $c$ is hidden by the fact that, thanks to $\det S=1$, $c=\tfrac{ad-1}b$.

On the contrary, the main formula in Theorem \ref{thm1} is expressed directly on the matrix $S$ and therefore admits an extension to the case $\det S\neq 1$. Note that this extension is highly non-trivial also in the real case $S\in M(2,\R)$.

In the present paper, we will limit ourself to the case $\det S=\pm 1$.  
We set
$$
M^\pm(2,\C):=\{S\in SL(2,\C),\ \det S=\pm 1\}.
$$

\begin{definition}\label{defnonca}
Let
\[
H=\int h(z)\vert\psi^\beta_z\rangle\langle\psi^\beta_z\vert\frac{dzd\bar z}{2\pi\hbar}.
\]
Define, for $S
\in M^\pm(2,\C),\ \det S\neq0,$ and $\Im\alpha>0$, the real $2\times 2$ matrix ${_\alpha T_S}$ by 
$$
{_\alpha T_S(z)}=(q_S^\alpha,p_S^\alpha)\in\R^2\mbox{ defined by }q_S^\alpha+\alpha_S(\alpha)p_S^\alpha=q^S+\alpha_S(\alpha)p^S,\ 
\begin{pmatrix}
q^S\\p^S
\end{pmatrix}:=S(z),\ z=(q,p).
$$
For any $S,\alpha$ such that $|V\cdot \alpha|<\infty,\ \Im(V\cdot \alpha)>0,\ V=S^{-1},\com{S}^{-1}, S^{-1}\com{S}$, 
we define the composition operator $\ccomp{S}$ acting on $H$ by
\be
\ccomp{S}H=\int 
 {_\alpha T_{\com{S}}}\# h_{\com{S}\cdot \alpha,\alpha}((-1)^{\frac{1-\det{S}}2}z)
\tfrac{\vert\psi^{S^{-1}\com{S}\cdot\alpha}
_{{_\alpha T_{S^{-1}\com{S}}}(z)}
\rangle\langle\psi^{\alpha}_z\vert}
{\langle\psi^{S^{-1}\com{S}\cdot\alpha}
_{{_\alpha T_{S^{-1}\com{S}}}(z)}
\vert I^{\frac{1-\det{S}}2}\psi^{\alpha}_z\rangle}
\frac{dzd\bar z}{2\pi\hbar},\label{defpm}
\ee
where $I$ is the parity operator defined on $L^2(\R)$ by $I\psi(x)=\psi(-x)$.
\end{definition}
When $\det{S}=1$, \eqref{defpm} is the same as the result of Theorem \ref{thm1} so that, in this case, $\ccomp{S}\cdot=U(S)^{-1}\cdot U(S)$.
When $\det{S}=-1$, the presence of the operator $I$ in the normalization constant $\frac 1
{\langle\psi^{S^{-1}\com{S}\cdot\alpha}
_{{_\alpha T_{S^{-1}\com{S}}}(z)}
´\vert I\psi^{\alpha}_z\rangle}$  and of the factor $(-1)^{\frac{1-\det{S}}2}$ in the argument of $ {_\alpha T_{\com{S}}}\# h_{\com{S}\cdot \alpha,\alpha}$ follows from the following  two arguments.

 First, we have seen right after its statement, that the key stone of the proof of Theorem \ref{thm1} was the fact that the normalization constant $L$ ensures  $U(S^{-1})|\psi^\alpha_z\rangle\langle\psi^\alpha_z|U(S)$
 to be a projector. 
 
 \noindent Let us remind that, for two operators $R,R'$, the Wigner function of the product $RR'$ is expressed as the twisted convolution of the Wigner functions of $R$ and $R'$. Namely:
 \be\label{twistconv0}
 W[RR'](z)=\int W[R](z-z')W[R'](z)e^{i\frac{z\wedge z'}\hbar} dz
 \ee
 
 \noindent Therefore, when $\det{S}=1$,  the requirement of being a projector can be  seen as following the fact that Wigner functions of pure states composed by canonical transforms satisfy the same equality than the original one, namely
 \be\label{twistconv}
 \int S\#W(z-z')S\#W(z)e^{i\frac{z\wedge z'}\hbar} dz'=S\#W(z)\Leftrightarrow
 \int W(z-z')W(z')e^{i\frac{z\wedge z'}\hbar} dz'=W(z)
 \ee
 since $\det{S}=1\Rightarrow S(z)\wedge S(z')=z\wedge z'$. 
 \noindent When $\det(S)=-1$, the left hand side of \eqref{twistconv} becomes
\be\label{twistconv2}
 \int S\#W(z-z')S\#W(z)e^{-i\frac{z\wedge z'}\hbar} dz'=S\#W(z)\Leftrightarrow
 \int W(z-z')W(-z')e^{i\frac{z\wedge z'}\hbar} dz'=W(z)
 \ee 
 leading to, if $R$ denotes the operator of Wigner function $W$, $\ccomp{S}RI\ccomp{S}R=\ccomp{S}R$. This shows easily that $I$ has to be introduced in $\langle\psi^{S^{-1}\com{S}\cdot\alpha}
_{{_\alpha T_{S^{-1}\com{S}}}(z)}
´\vert I\psi^{\alpha}_z\rangle$.

Secodnly, in the course of the proof of Theorem \ref{thm1}, we have used the equality 
\be\label{eqweyl}
U(S^{-1})e^{i\frac{z\wedge Z}\hbar}U(S)=e^{i\frac{z\wedge S(Z)}\hbar}
=e^{i\frac{S^{-1}(z)\wedge Z}\hbar},\ Z=\binom{x}{-i\hbar\frac d{dx}},
\ee
due to the fact that $S$ is canonical.
When $\det{S}=-1$, \eqref{eqweyl} becomes
$$
U(S^{-1})e^{i\frac{z\wedge Z}\hbar}U(S)=e^{i\frac{z\wedge S(Z)}\hbar}
=e^{-i\frac{S^{-1}(z)\wedge Z}\hbar},
$$
responsible for the change $z\to -z$ in $\vert\psi^{S^{-1}\com{S}\cdot\alpha}
_{{_\alpha T_{S^{-1}\com{S}}}(z)}
\rangle\langle\psi^{\alpha}_z\vert$ and therefore in the argument of $ {_\alpha T_{\com{S}}}\# h_{\com{S}\cdot \alpha,\alpha}$ by change of variable in the integration in \eqref{defpm}.

Note again that, on the contrary of the symplectic case, $\ccomp{S}$ is not in general a conjugation.
Nevertheless, since $\ccomp{S}H$ has the form 
$\ccomp{S}H=\int f(z)|\psi^{\alpha'(\alpha)}_{z'(z)}\rangle\langle\psi^{\alpha}_{z}|dz$, 
one can  extend $\ccomp{S}$, asin the conjugation case, to more general operator than the Toeplitz class and define $\ccomp{S}\ccomp{S'}$ by the same formula as in definition \ref{defnonca} after first replacing  $(z,\alpha)$ by $({{_\alpha T_{S^{-1}\com{S}}}(z)},{S^{-1}\com{S}\cdot\alpha})$ and then multiplying by the weight $$
\left.\tfrac{{_\alpha T_{\com{S}}}\# h_{\com{S}\cdot \alpha,\alpha}(z)}
{{\langle\psi^{S^{-1}\com{S}\cdot\alpha}
_{{_\alpha T_{S^{-1}\com{S}}}(z)}
\vert\psi^{\alpha}_z\rangle}}
\right|_{S=S'}.$$

With this definition of $\ccomp{S'}
$, $\ccomp{\cdot}$  is a representation of $M^\pm(d,\C)$:
\begin{theorem}
$$
\ccomp{S'}\ccomp{S}=\ccomp{S'S}\ \mbox{ for all } S,S'\mbox{ in }M^\pm(2,\C).
$$ 
\end{theorem} 
 
As a significant example useful in the next section, let us consider the case $S=\begin{pmatrix}0&i\\-i&0\end{pmatrix}$ computed in the second table of Section \ref{annb}. We get, in the case $\alpha=i$,
\be\label{exas}
\ccomp{S}H=\int 
h(z)
\vert\psi^{i}
_{-z}
\rangle\langle\psi^{i}_z\vert\frac{dzd\bar z}{2\pi\hbar},
\ee
since $\langle\psi^i_{-z}\vert I\psi^i_z\rangle=\langle\psi^i_{-z}\vert\psi^i_{-z}\rangle=1$.

In other words,  $\ccomp{S}H$ is the quantization of the symbol (with a slight abuse of notation)
\be\label{exassymb}
\sigma^{off}[\ccomp{S}H]((z,i),(z',i))
=
h(z)\delta(z'+z).
\ee
The case $S=\begin{pmatrix}0& -i\\ i&0\end{pmatrix}$ can be treated the same way and leads to, thanks to the same table,
Both cases $S=\begin{pmatrix}0&i\\-i&0\end{pmatrix},\ \begin{pmatrix}0& -i\\ i&0\end{pmatrix}$  are shown to be underlying the exchange operator in the framework of quantum spin-statistics as extensively studied in \cite{tp4}.
\section{Link with complex symplectic geometry}\label{sympgeo}
$\ec{\alpha}{(p,q)}$ is in fact a complex WKB state associated to the Lagrangian (complex)
$$
{_\alpha\Lambda_{(p,q)}}=\{(x,\nabla_x(- \frac{(x-q)^2}{2\alpha}+px)=(x,-\frac{x-q}\alpha+p),\ x\in\R\}.
$$
Note that ${_\alpha\Lambda_{(p,q)}}\cap\R=\{(q,p)\}$ and $(x,\xi)\in{_\alpha\Lambda_{(p,q)}}\Longleftrightarrow x+\alpha\xi=q+\alpha p$ so that
$$
{_\alpha\Lambda_{(p,q)}}=\{(x,\xi)/x+\alpha\xi=q+\alpha p\}.
$$
Therefore
\begin{eqnarray}
S^{-1}({_\alpha\Lambda_{(p,q)}})&=&
\{q^{S^{-1}}(x,\xi)+\alpha p^{S^{-1}}(x,\xi)=q+\alpha p\}\nn
&=&{_{S\cdot\alpha}}\Lambda_{{_\alpha T_S}(q,p)}\nonumber
\end{eqnarray}

\section{Higher dimensions}\label{highdim}
The whole discussion above easily generalizes in higher dimension $n$.

To $\alpha\in M_n(\C),\ \alpha^T=\alpha,\ \Im{\alpha}>0$ and $z\in T^*\R^n\sim\R^{2n}$ we associate the vector in $L^2(\R^n,dx)$
$$
\ec{\alpha}{z}(x)
=(\pi\hbar)^{-\frac n4}(\det\Im(\alpha^{-1}))^{\frac14}
e^{-\frac i\hbar(x-q)\alpha^{-1}(x-q)+i\frac{p\cdot x}\hbar-i\frac{q\cdot p}{2\hbar}}
$$
Formulas \eqref{alphaalphaprime}-\eqref{betabetaprime} become, for  $\alpha\in M_n(\C),\ \alpha^T=\alpha,\ \Im{\alpha}>0$
\begin{eqnarray}
\int_\C h_\alpha(z)\vert\ec{\alpha}{z}\rangle\langle\ec{\alpha}{z}\vert\frac{dzd\bar z}{2\pi\hbar}
&=&
\int_\C h_{\alpha,\alpha'}(z)\vert\ec{\alpha'}{z}\rangle\langle\ec{\alpha'}{z}\vert\frac{dzd\bar z}{2\pi\hbar}\label{alphaalphaprimen}\\
&\Updownarrow&\nonumber\\
h_{\alpha,\alpha'}&=&e^{-i\frac\hbar4\nabla_\xi(\alpha-\alpha')\nabla_\xi+i\frac\hbar4\nabla_x(\frac1\alpha-\frac1{\alpha'})\nabla_x}h_{\alpha'}\label{betabetaprimen}
\end{eqnarray}
for $ \Im{(\alpha-\alpha')}>0.\ \Im{(\frac1\alpha-\frac1{\alpha'})}<0$.

We denote for $S= \begin{pmatrix}
A&B\\C&D
\end{pmatrix}\in SP(n,\C)$, $z=(q,p)\in\R^{2n}$ and $\alpha\in M_n(\C), \alpha^T=\alpha, \Im{\alpha}>~0$, 
$$
\com{S}= \begin{pmatrix}
\bar  A&\bar B\\\bar C&\bar D
\end{pmatrix},\ S (z)=\begin{pmatrix}
A&B\\C&D
\end{pmatrix}\begin{pmatrix}
q\\p
\end{pmatrix}\mbox{ and }
S\cdot\alpha=({C\alpha+D})^{-1}({A\alpha+B}). 
$$
Moreover, $\wedge$ will denote the symplectic form pn $T^*\R$, $z\wedge z'=p\cdot q'-q\cdot p'$.

By easy computations of Gaussian integrals and the same arguments as in the proof of Proposition \ref{lem2} we get the following result.
\begin{proposition}[Proposition \ref{lem2} in dimension $n$]\label{lem2n}\ 

Let $S=\left(\begin{matrix}
A&B\\
C&D
\end{matrix}\right)\in Spn(2n,\C)$ and let  $U=U(S)$ be defined through its integral kernel
\bea
U(x,y)&=&\frac1{\sqrt{2\pi\hbar\det{B}}} e^{-\frac i{2\hbar}\left(xDB^{-1}x-2xB^{-1}y+yB^{-1}Ay\right)}\ \ \ \ \ \ \ \ \ \ \det B\neq 0\nonumber\\
\int U(x,y)e^{-i\frac{y\cdot\xi}\hbar}\tfrac{dy}{(2\pi\hbar)^{\frac n2}}&=&
\frac{i^{-\frac n2}}{\sqrt{2\pi\hbar\det{A}}} e^{-\frac i{2\hbar}\left(xCA^{-1}x-2xA^{-1}\xi+\xi A^{-1}B\xi\right)}
\ \ \ \ \ \ \ \ \ \ \   \det A\neq 0\nonumber
\eea
(note that $\det S=1\Rightarrow\det B\neq 0\mbox{ or }\det A\neq 0$).

Then, for any $S,\alpha$ such that $|\alpha_S(\alpha)|<\infty,\ \Im\alpha_S(\alpha)>0$ and any $z\in\R^{2n}$, 
\[
U(S)\ec{\alpha}{z}=e^{i\frac {S(z)\wedge {_\alpha T}_S(z)}{2\hbar}}
\ec{S\cdot\alpha}{{_\alpha T}_S(z)}.
\]
where  the real $2\times 2$ matrix ${_\alpha T_S}$ by 
$$
{_\alpha T_S(z)}=(p_S^\alpha,q_S^\alpha)\in\R^2\mbox{ defined by }q_S^\alpha+S\cdot\alpha \ p_S^\alpha=q^S+S\cdot\alpha \ p^S,\ 
\begin{pmatrix}
p^S\\q^S
\end{pmatrix}:=S(z),\ z=(q,p).
$$
\end{proposition}
\begin{proof}
The proof consists in elementary Gaussian integrals computations. We perform it in the case  $\det B\neq 0$, the case $\det A\neq 0$ being the same by Fourier transform.

Denoting $Z$ the vector $Z=\binom{x}{-i\hbar\frac{d}{dx}}$ we get
\begin{eqnarray}
U(S)\ec{\alpha}{z}
&=&U(S)e^{\frac i\hbar z\wedge Z}\ec{\alpha}{0}\nonumber\\
&=&U(s)e^{\frac i\hbar z\wedge Z}U(S^{-1})U(S)\ec{\alpha}{0}\nonumber\\
&=&
e^{\frac i\hbar S(z)\wedge Z}U(S)\ec{\alpha}{0}.\label{wcs}
\end{eqnarray}
Now, denoting
\begin{eqnarray}
C&=&(\det{\Im{((\alpha)^{-1})}})^{-\frac14}(\pi\hbar)^{-\frac n4}(\det B)^{-\frac12}(\pi\hbar)^{-\frac n2},\nonumber\\
C'&=&(\det(B^{-1}A+\alpha^{-1})^{-\frac12}(\pi\hbar)^{\frac n2},\nonumber\\ C"&=&(\det{\Im{((S\cdot\alpha)^{-1})}})^{\frac14}(\pi\hbar)^{\frac n4},\nonumber
\end{eqnarray}

we have
\begin{eqnarray}
U(S)\ec{\alpha}{0}(x)
&=&
C\int e^{-\frac i{2\hbar}\left(xDB^{-1}x-2xB^{-1}y+yB^{-1}Ay\right)}\ec{\alpha}{0}(y)dy \nonumber\\
&=&
C\int e^{-\frac i{2\hbar}\left(xDB^{-1}x-2xB^{-1}y+yB^{-1}Ay\right)}
e^{-\frac i{2\hbar}y\alpha^{-1}y}dy\nonumber\\
&=&
CC'e^{-\frac i{2\hbar}x\left(DB^{-1}-(B^{-1})^T(B^{-1}A+\alpha^{-1})^{-1}B^{-1}\right)x}\nonumber\\
&=&
CC'
e^{-\frac i{2\hbar}x\left(DB^{-1}-(AB^T+B\alpha^{-1}B^T)^{-1}\right)x}\nonumber\\
&=&
CC'
e^{-\frac i{2\hbar}x\left(DB^{-1}(AB^T+B\alpha^{-1}B^T)-1)(AB^T+B\alpha^{-1}B^T)^{-1}\right)x}\nonumber\\
&=&
CC'
e^{-\frac i{2\hbar}x\left(DB^{-1}AB^T+D\alpha^{-1}B^T-1)(AB^T+B\alpha^{-1}B^T)^{-1}\right)x}\nonumber\\
&=&
CC'
e^{-\frac i{2\hbar}x\left((B^{-1})^TD^TAB^T+D\alpha^{-1}B^T-1)(AB^T+B\alpha^{-1}B^T)^{-1}\right)x}\nonumber\\
&=&
CC'
e^{-\frac i{2\hbar}x\left((B^{-1})^T(D^TA-I)B^T+D\alpha^{-1}B^T)(AB^T+B\alpha^{-1}B^T)^{-1}\right)x}\nonumber\\
&=&
CC'
e^{-\frac i{2\hbar}x\left((B^{-1})^T(B^TC)B^T+D\alpha^{-1}B^T)(AB^T+B\alpha^{-1}B^T)^{-1}\right)x}\nonumber\\
&=&
CC'
e^{-\frac i{2\hbar}x\left((C+D\alpha^{-1})B^T)(AB^T+B\alpha^{-1}B^T)^{-1}\right)x}\nonumber\\
&=&
CC'
e^{-\frac i{2\hbar}x\left(C+D\alpha^{-1}(A+B\alpha^{-1})^{-1}\right)x}\nonumber\\
&=&
CC'
e^{-\frac i{2\hbar}x\left(C\alpha+D(A\alpha+B)^{-1}\right)x}\nonumber\\
&=&
CC'
e^{-\frac i{2\hbar}x\left((A\alpha+B)(C\alpha+D)^{-1}\right)^{-1}x}\nonumber\\
&=&
CC'
e^{-\frac i{2\hbar}x\left(S\cdot\alpha\right)^{-1}x}.\nonumber\\
&=&CC'C"\ec{S\cdot\alpha}{0}\nonumber\\
&=&
(\det(A+B\alpha^{-1})^{-\frac12}\left(\tfrac{\det\Im(S\cdot\alpha)^{-1}}{\det\Im(\alpha^{-1})} \right)^{\frac14}\ec{S\cdot\alpha}{0},\nonumber
\end{eqnarray}
so that, by \eqref{wcs},
\begin{eqnarray}
U(S)\ec{\alpha}{0}&=&(\det(A+B\alpha^{-1})^{-\frac12}\left(\tfrac{\det\Im(S\cdot\alpha)^{-1}}{\det\Im(\alpha^{-1})} \right)^{\frac14}\ec{S\cdot\alpha}{0}\nonumber
\end{eqnarray}
The end of the proof is exactly the same as in dimension 1.
\end{proof}
%
The (quasi) same proof as for Theorem \ref{thm1} leads to the following one, verbatim the same.
\begin{theorem}[
Theorem \ref{thm1} in dimension $n$]\label{thm1n}\ 

Let
\[
H=\int h(z)\vert\ec{\alpha}{z}\rangle\langle\ec{\alpha}{z}\vert\frac{dzd\bar z}{2\pi\hbar}.
\]
Then, for any $S,\alpha$ such that $|V\cdot \alpha|<\infty,\ \Im(V\cdot \alpha)>0,\ V=S^{-1},\com{S}^{-1}, S^{-1}\com{S}$, 

\[
U(S)^{-1}HU(S)=\int 
h_{\com{S}\cdot \alpha,\alpha}(_\alpha T_{\com{S}}z)
|\det{_\alpha T_\com{S}}|
{_\alpha \fract_S}
e^{i\frac{
S^{-1}({_\alpha T_{\com{S}}}(z))\wedge({_\alpha T_{S^{-1}\com{S}}(z)-z
})}\hbar}
\frac{\vert\psi^{S^{-1}\com{S}\cdot\alpha}
_{{_\alpha T_{S^{-1}\com{S}}}(z)}
\rangle\langle\psi^{\alpha}_z\vert}
{\langle\psi^{S^{-1}\com{S}\cdot\alpha}
_{{_\alpha T_{S^{-1}\com{S}}}(z)}
\vert\psi^{\alpha}_z\rangle}
\frac{dzd\bar z}{2\pi\hbar},
\]
where $h_{\com{S}\cdot \alpha,\alpha}$ is defined by \eqref{alphaalphaprimen}-\eqref{betabetaprimen}.
\end{theorem}

\vskip 1cm
\begin{appendix}
\section{Weyl}\label{w}
Easy computations of Gaussian integrals show the following result.
\begin{lemma}\label{husweyl}
\bea
\sigma_{Weyl}^{|\psi^\alpha_z><\psi^{\alpha'}_{z'}|}(x,\xi)=\nn\left.
\sqrt{\frac{2\sqrt{\Re\beta\Re\beta'}}{\beta+\bbetap}}\frac1\hbar
e^{-\frac{\beta\bbetap}{2(\beta+\bbetap)\hbar}(q+q'-2x)^2}
e^{-\frac{1}{2(\beta+\bbetap)\hbar}(p+p'-2\xi)^2}
e^{i[(p-p')x-(p+p'-2\xi)(\beta(x-q)-\bbetap(x-q'))/(\beta+\bbetap)\hbar]}\right|_{\substack{\beta=i\alpha\\ \beta'=i\alpha'}}\nonumber
\eea
In particular when $z=z',\ \alpha=\alpha'>0$,
\[
\sigma_{Weyl}^{|\psi^\alpha_z><\psi^{\alpha}_{z}|}(x,\xi)=
\frac1{\pi\hbar}e^{-\frac i\alpha\frac{(q-x)^2}\hbar}
e^{i\alpha\frac{(p-\xi)^2}{\hbar}}
\]
\end{lemma}

\begin{corollary}\label{hushus}
\[
\int_\C h_\alpha(z)\vert\pbz\rangle\langle\pbz\vert\frac{dzd\bar z}{2\pi\hbar}
=
\int_\C h_{\alpha,\alpha'}(z)\vert\psi^{\alpha'}_z\rangle\langle\psi^{\alpha'}_z\vert\frac{dzd\bar z}{2\pi\hbar}
\]
if and only if 
\be\label{laplac}
h_{\alpha,\alpha'}=e^{-i\frac\hbar4(\alpha-\alpha')\Delta_\xi+i\frac\hbar4(\frac1\alpha-\frac1{\alpha'})\Delta_x}h_{\alpha'},\ 
\Im{(\alpha-\alpha')}>0.\ \Im{(\frac1\alpha-\frac1{\alpha'})}<0.
\ee
\end{corollary}
\begin{proof}
From Lemma \ref{husweyl} we get that
\[
e^{-i\alpha\hbar\frac{\Delta_x}4+\frac i\alpha\hbar\frac{\Delta_\xi}4}h(x,\xi)
=
e^{-i\alpha'\hbar\frac{\Delta_x}4+\frac i{\alpha'}\hbar\frac{\Delta_\xi}4}
h_{\alpha'}(x,\xi)
\]
and the corollary follows.
\end{proof}
%

\end{appendix}



\end{document}